\documentclass[12pt,twoside]{article}
\usepackage{amsmath,amsthm,amssymb,mathabx}

\usepackage[a4paper ,left=20mm,right=20mm,top=30mm,bottom=30mm]{geometry}
\usepackage{imakeidx}
\usepackage{xr-hyper} 
\usepackage{xr}
\usepackage[all]{xy}

\usepackage{cleveref,fancyhdr}

\usepackage{seqsplit}
\usepackage{xstring}
\usepackage[xcdraw]{xcolor}
\usepackage{todonotes}

\definecolor{mycolor}{rgb}{0.122, 0.435, 0.698}

\makeatletter
\DeclareOldFontCommand{\rm}{\normalfont\rmfamily}{\mathrm}
\DeclareOldFontCommand{\sf}{\normalfont\sffamily}{\mathsf}
\DeclareOldFontCommand{\tt}{\normalfont\ttfamily}{\mathtt}
\DeclareOldFontCommand{\bf}{\normalfont\bfseries}{\mathbf}
\DeclareOldFontCommand{\it}{\normalfont\itshape}{\mathit}
\DeclareOldFontCommand{\sl}{\normalfont\slshape}{\@nomath\sl}
\DeclareOldFontCommand{\sc}{\normalfont\scshape}{\@nomath\sc}
\makeatother

\usepackage[shortlabels]{enumitem}
\setlist{nolistsep} %

\newlist{asslist}{enumerate}{1} 
\setlist[asslist]{label=(\roman*), ref=\thethmT(\roman*)}
\crefalias{asslistenumi}{Assumption} 

\newlist{asslisttief}{enumerate}{1} 
\setlist[asslisttief]{label=(\roman*), ref=\thethmlistT(\roman*)}
\crefalias{asslisttiefenumi}{Property}

\newlist{thmlist}{enumerate}{1} 
\setlist[thmlist]{label=(\alph*), ref=\thethmT(\alph*)}

\usepackage{tikz} 

\usepackage{xcolor} 
\definecolor{ocre_old}{RGB}{243,102,25} 
\definecolor{ocre}{rgb}{0.122, 0.435, 0.698}
\usepackage{amsmath,amsfonts,amssymb,amsthm} 

\makeatletter
\newtheoremstyle{ocrenumbox}
{0pt}
{0pt}
{\sl}
{}
{\small\bf\sffamily\color{ocre}}
{\;}
{0.25em}
{\small\sffamily\color{ocre}\thmname{#1}\nobreakspace\thmnumber{\@ifnotempty{#1}{}\@upn{#2}}
	\thmnote{\nobreakspace\the\thm@notefont\sffamily\bfseries\color{black}---\nobreakspace#3.}} 

\newtheoremstyle{ocrenumhypbox}
{0pt}
{0pt}
{}
{}
{\small\bf\sffamily\color{ocre}}
{\;}
{0.25em}
{\small\sffamily\color{ocre}\thmname{#1}\nobreakspace\thmnumber{\@ifnotempty{#1}{}\@upn{#2}}
	\thmnote{\nobreakspace\the\thm@notefont\sffamily\bfseries\color{black}---\nobreakspace#3.}} 

\newtheoremstyle{blacknumex}
{5pt}
{5pt}
{\sl}
{} 
{\small\bf\sffamily}
{\;}
{0.25em}
{\small\sffamily{\tiny\ensuremath{\blacksquare}}\nobreakspace\thmname{#1}\nobreakspace\thmnumber{\@ifnotempty{#1}{}\@upn{#2}}
	\thmnote{\nobreakspace\the\thm@notefont\sffamily\bfseries---\nobreakspace#3.}}

\newtheoremstyle{blacknumbox} 
{0pt}
{0pt}
{\normalfont}
{}
{\small\bf\sffamily}
{\;}
{0.25em}
{\small\sffamily\thmname{#1}\nobreakspace\thmnumber{\@ifnotempty{#1}{}\@upn{#2}}
	\thmnote{\nobreakspace\the\thm@notefont\sffamily\bfseries---\nobreakspace#3.}}

\newtheoremstyle{ocrenum}
{5pt}
{5pt}
{\sl}
{}
{\small\bf\sffamily\color{ocre}}
{\;}
{0.25em}
{\small\sffamily\color{ocre}\thmname{#1}\nobreakspace\thmnumber{\@ifnotempty{#1}{}\@upn{#2}}
	\thmnote{\nobreakspace\the\thm@notefont\sffamily\bfseries\color{black}---\nobreakspace#3.}} 
\makeatother

\theoremstyle{ocrenumbox}
\newtheorem{thmT}{Theorem}[section]
\newtheorem{theoT}{Theorem}
\newtheorem{theoremeT}[thmT]{Theorem}
\newtheorem{lemT}[thmT]{Lemma}

\theoremstyle{ocrenumhypbox}
\newtheorem{hypT}[thmT]{Hypothesis}
\theoremstyle{blacknumex}

\theoremstyle{blacknumbox}
\newtheorem{definitionT}[thmT]{Definition}
\newtheorem{notationT}[thmT]{Notation}

\theoremstyle{ocrenum}

\newtheorem{propT}[thmT]{Proposition}

\newtheorem{corollaryT}[thmT]{Corollary}

\RequirePackage[framemethod=default]{mdframed} 

\newmdenv[skipabove=7pt,
skipbelow=7pt,
backgroundcolor=black!5,
linecolor=ocre,
innerleftmargin=5pt,
innerrightmargin=5pt,
innertopmargin=5pt,
leftmargin=0cm,
rightmargin=0cm,
innerbottommargin=5pt]{tBox}

\newmdenv[skipabove=7pt,
skipbelow=7pt,
rightline=false,
leftline=true,
topline=false,
bottomline=false,
backgroundcolor=ocre!10,
linecolor=ocre,
innerleftmargin=5pt,
innerrightmargin=5pt,
innertopmargin=5pt,
innerbottommargin=5pt,
leftmargin=0cm,
rightmargin=0cm,
linewidth=4pt]{eBox}	

\newmdenv[skipabove=7pt,
skipbelow=7pt,
rightline=false,
leftline=true,
topline=false,
bottomline=false,
linecolor=ocre,
innerleftmargin=5pt,
innerrightmargin=5pt,
innertopmargin=0pt,
leftmargin=0cm,
rightmargin=0cm,
linewidth=4pt,
innerbottommargin=0pt]{dBox}	

\newmdenv[skipabove=7pt,
skipbelow=7pt,
rightline=false,
leftline=true,
topline=false,
bottomline=false,
linecolor=gray,
backgroundcolor=black!5,
innerleftmargin=5pt,
innerrightmargin=5pt,
innertopmargin=5pt,
leftmargin=0cm,
rightmargin=0cm,
linewidth=4pt,
innerbottommargin=5pt]{cBox}

\newenvironment{theo}{\begin{tBox}\begin{theoT}}{\end{theoT}\end{tBox}}

\newenvironment{lem}{\begin{dBox}\begin{lemT}}{\end{lemT}\end{dBox}}	
\newenvironment{prop}{\begin{dBox}\begin{propT}}{\end{propT}\end{dBox}}

\newenvironment{cor}{\begin{dBox}\begin{corollaryT}}{\end{corollaryT}\end{dBox}}	

\makeatletter
\renewcommand{\@seccntformat}[1]{\llap{\textcolor{ocre}{\csname the#1\endcsname}\hspace{1em}}} 
\renewcommand{\section}{\@startsection{section}{1}{\z@}
	{-4ex \@plus -1ex \@minus -.4ex}
	{1ex \@plus.2ex }
	{\normalfont\large \bf \color{ocre}}}
\renewcommand{\subsection}{\@startsection {subsection}{2}{\z@}
	{-3ex \@plus -0.1ex \@minus -.4ex}
	{0.5ex \@plus.2ex }
	{\normalfont\large\bf\color{ocre} }}

\newcommand{\ul}{{\underline {l}}}

\def\titlerunning#1{\gdef\titrun{#1}}
\makeatletter
\def\author#1{\gdef\autrun{\def\and{\unskip, }#1}\gdef\@author{#1}}
\def\address#1{{\def\and{\\\hspace*{18pt}}\renewcommand{\thefootnote}{}%
		\footnote {#1}}%
	\markboth{Cabanes-Sp\"ath }{\titrun}}
\makeatother
\def\email#1{e-mail: #1}
\def\subjclass#1{{\renewcommand{\thefootnote}{}%
		\footnote{\emph{Mathematics Subject Classification (2010):} #1}}}

\newcommand{\inv}{^{-1}}

\theoremstyle{definition}

\newtheorem{rem}[thmT]{Remark}

{\color{ocre}}
\theoremstyle{plain}

\theoremstyle{definition}

\numberwithin{equation}{section}
\numberwithin{table}{section}

\newcommand{\Id}{\operatorname {Id}}

\newcommand{\wc}{\widecheck}

\newcommand{\bC}{{\mathbf C}}

\newcommand{\bT}{{\mathbf T}}

\newcommand{\bB}{{\mathbf B}}

\newcommand{\bG}{{{\mathbf G}}}

\newcommand{\bH}{{{\mathbf H}}}

\newcommand{\bHnull}{{\bH_0}}

\newcommand{\HF}{{{\bH}^F}}

\newcommand{\Irr}{\mathrm{Irr}}

\newcommand{\SL}{\operatorname{SL}}

\newcommand{\GL}{\operatorname{GL}}
\newcommand{\SO}{\operatorname{SO}}

\newcommand{\ZZ}{\ensuremath{\mathbb{Z}}}

\newcommand{\ov}{\overline }

\newcommand{\hh}{\mathbf h }

\newcommand{\Cent}{\ensuremath{{\rm{C}}}}
\newcommand{\NNN}{\ensuremath{{\mathrm{N}}}}
\newcommand{\ZZZ}{\ensuremath{{\mathrm{Z}}}}
\newcommand{\Sym}{{\mathcal{S}}}

\newcommand{\uE}{{\underline E}}
\newcommand{\HE}{{\bH\underline E}}

\def\restr#1|#2{\left.#1\right\rceil_{#2}}

\usepackage{imakeidx}
\makeindex[columns=3,intoc]

\def\II#1@#2{\index{#1@$#2$}{{\color{ocre}#2}}}

\newcommand{\calC}{\mathcal C}

\newcommand{\al}{{\alpha}}
\newcommand{\eps}{{\epsilon}}

\newcommand{\spannh}{\spann<h_0>}

\newcommand{\FF}{{\mathbb{F}}}

\newcommand{\si}{\ensuremath{\sigma}}

\newcommand{\GF}{{{\bG^F}}}

\newcommand{\HFnull}{{{{\bH}^{F_0}}}}

\makeatletter
\def\Set#1{\Set@h#1@}
\def\Lset#1{\Lset@h#1@}
\def\Set@h#1|#2@{\left\{\left.#1\vphantom{#2}\hskip.1em\,\right|\,\relax #2\right\}}
\def\Lset@h#1@{\left\{#1\right\}}
\def\CALC#1{\CALC@h#1@}
\def\CALC@h#1|#2@{\calC^{#1}(#2)}
\def\CALCrad#1{\CALCrad@h#1@}
\def\CALCrad@h#1|#2@{\calC_\radic^{#1}(#2)}

\def\CALCNC#1{\CALCNC@h#1@}
\def\CALCNC@h#1|#2@{\calC_{\radic,nc}^{#1}(#2)}

\def\restr#1|#2{\left.#1\right\rceil_{#2}}
\def\spann<#1>{\left\langle#1\right\rangle}
\def\spa#1{\left\langle#1\right\rangle}

\def\Spann<#1>{\Spann@h#1@}
\def\Spann@h#1|#2@{\left\langle\left.#1\vphantom{#2}\hskip.1em\right.\mid\relax #2 \right\rangle}
\def\Set#1{\Set@h#1@}
\def\Set@h#1|#2@{\left\{\left.#1\vphantom{#2}\hskip.1em\,\right.
	\mid\relax #2\right\}}
\def\set#1{\set@h#1@}
\def\set@h#1@{\left\{#1\right\}}
\def\spann<#1>{\left\langle#1\right\rangle}
\makeatother

\newcommand{\AHs}{\mathrm{A}_\bH(s)}

\newcommand{\Out}{\ensuremath{\mathrm{Out}}}

\newcommand{\Z}{\operatorname Z}

\newcommand{\pr}{\operatorname{pr}}

\renewcommand{\O}{\operatorname{O}}

\newcommand{\ra}{\rightarrow}
\newcommand{\lra}{\longrightarrow}

\newcommand{\cC}{\mathcal C}

\newcommand{\wrt}{{with respect to\ }}

\titlerunning{On semisimple classes and component groups in type ${\mathsf{D}}$ }
\title{On semisimple classes and component groups in type ${\mathsf{D}}$\\
	{\small \textit{Dedicated to Pham Huu Tiep on the occasion of his 60th birthday} }}
\author{Marc Cabanes \and Britta Sp\"ath  \thanks{
		Some of this research was conducted in the framework of the research training group
		\emph{GRK 2240: Algebro-Geometric Methods in Algebra, Arithmetic and Topology}, funded by the DFG. 
		The authors would like to thank the Isaac Newton Institute for Mathematical Sciences for support during the programme 
		\textit{ Groups, Representations and Applications}, when work on this paper was undertaken. The Institute is supported by EPSRC Grant Number EP/R014604/1.}
}
\begin{document} 
	\maketitle
	
	\abstract{In adjoint simple algebraic groups $\bH$ of type $\mathsf{D}$ we show that for every semisimple element $s$, its centralizer splits over its identity component, i.e. $\Cent_{\bH}(s)=\Cent_{\bH}(s)^\circ\rtimes \wc A $ for some complement $\wc A$ with strong stability properties. We derive several consequences about the action of automorphisms on semisimple conjugacy classes. This helps to parametrize characters of the finite groups $\mathsf{D}_{l,\text{sc}}(q)$ and $^2\mathsf{D}_{l,\text{sc}}(q)$ and describe the action of automorphisms on them. It is also a contribution to the final proof of the McKay conjecture for the prime 3, see \cite{TypeD2}{}.}
	\address{M.C.: CNRS, Institut de Math\'ematiques Jussieu-Paris Rive Gauche, Place Aur\'elie Nemours, 75013 Paris, France.
		\email{cabanes@imj-prg.fr}\\
		\indent	\ \ \ \ B. S.: School of Mathematics and Natural Sciences
		University of Wuppertal, Gau\ss str. 20, 42119 Wuppertal, Germany, \email{bspaeth@uni-wuppertal.de}}


	\section*{Introduction}

	One of the many outstanding ideas in Deligne-Lusztig theory was to point out the relevance of a dual group $\bG^*$ to the study of representations of a finite group of Lie type $G=\bG^F$ (see for instance \cite[Ch. 11, 12]{DiMi}). In particular the set $\Irr(G)$ of irreducible (complex) characters is partitioned according to semisimple elements of $G^*=\bG^*{}^F$ up to $G^*$-conjugacy (``rational series") or $\bG^*$-conjugacy (``geometric series"). This explains how the study of $\Irr(G)$, for instance as an $\Out(G)$-set, may lead to purely group-theoretic questions about $G^*$ and its semisimple classes.

	 Together with Lusztig's Jordan decomposition of characters, rational series help to study characters in terms of unipotent characters of centralizers of semisimple elements $\Cent_{G^*}(s)$. It is a well known fact that the rational classes in a given geometric class are parametrized via the component group $A_{}(s):=\Cent_{\bG^*}(s)/\Cent_{\bG^*}^\circ(s)$ of the algebraic centralizer $\Cent_{\bG^*}(s)$. The various lifts in $\Cent_{\bG^*}(s)$ of a fixed element of $A_{}(s)$ provide different rational structures on the centralizer $\Cent_{\bG^*}(s)$ and different finite groups whose unipotent characters are then in correspondence with the characters of a given rational series. The results presented here are applied in \cite{TypeD2} to establish a Jordan decomposition compatible with automorphisms.
	
	The study of $\Irr(\GF)$ is probably most difficult when $\bG$ is a simply connected simple group and its type is $\mathsf{D}$. The difficulties come from the graph automorphism and a possibly non-cyclic fundamental group. However the dual group $\bG^*$ in that type, denoted here as $\bH$, has a couple of remarkable features with regard to semisimple classes and their centralizers. In particular we single out, for $s$ a semisimple element of $\bH$, a semidirect decomposition $$\Cent_{\bH}(s)=\Cent_{\bH}^\circ(s)\rtimes \wc A$$ where the complement $\wc A$ has a strong stability with regard to Frobenius endomorphisms and graph automorphisms.
	In \cite{TypeD2}, the elements of $\wc A$ are crucial to parametrize the different rational structures on $\Cent_{\bG^*}(s)$ and the action of Out$(\bG^F)$ on rational series of characters. This then also contributes to the verification of McKay's conjecture for the prime 3 in \cite{TypeD2}.

In type different from $\mathsf{D}$ our questions in representation theory have been solved by simpler methods not requiring this study of adjoint groups.
	It could be interesting however to prove results similar to Theorem A below in other types, at least classical. See \cite[3.2.1]{EngVers} for a property a bit weaker than the semidirect product structure.
	Theorem~A also shows that Condition (*) introduced in \cite[Def. 2]{DM21} is satisfied for the groups $\Cent_{\bG^*}(s)$ and all semisimple elements $s$, whenever $\bG$ is of type $\mathsf{D}$, thus establishing the first (minor) part of \cite[Conjecture 7]{DM21}. 
	
\noindent	{\bf Acknowledgements:} We thank Gunter Malle and the anonymous referees for remarks that helped improving the clarity of the paper. We also thank Alexei Galt for pointing the reference \cite{G17}.
	
	\section{Notations and main theorems}
	
	Let $\bH_0$ be a reductive group over the algebraically closed field $\FF$ of positive characteristic $p$. Assume $\bH_0$ is simple simply connected. Let $ \bB_0\geq \bT_0$ some Borel subgroup and maximal torus of $\bH_0$. We denote by $X(\bT_0)\supseteq \Phi\supseteq \Delta$ the associated roots and basis of the root system. 
	Let $x_\alpha\colon \FF\to \bH_0$ be the unipotent 1-parameter subgroup associated with $\alpha\in \Phi$. We set $F_p\colon \bH_0\to \bH_0$ defined by $F_p(x_\alpha(t))=x_\alpha(t^p)$ for any $\alpha\in\Phi$ and $t\in \FF$. Let $\Gamma$ be the group of algebraic automorphisms of $\bH_0$ that satisfy $x_{\eps\delta}(t)\mapsto x_{\eps\delta '}(t)$ for any $t\in\FF$, $\eps\in \{1,-1\}$, $\delta\in \Delta $ and where $\delta\mapsto\delta '$ is an automorphism of the Dynkin diagram. Note that $\Gamma$ is cyclic of order  2 whenever $\bH_0$ is of type $\mathsf{A}_l$ ($l\geq 2$), $\mathsf{D}_l$ ($l\geq 5$) or $\mathsf{E}_6$, dihedral of order 6 for type $\mathsf{D}_4$ and trivial for other types. Let $E :=\spann<F_p>\Gamma$ be the group of abstract automorphisms of $\bH_0$ generated by $\Gamma$ and $F_p$.
	
	Let $\bH:=\bH_0/\Z(\bH_0)$ and $$\pi\colon \bH_0\to\bH$$ the quotient map. Then $E$ clearly acts on $\bH$ and
	we can form the semidirect product $\bH\rtimes E $. 
	
	Let $s$ be a semisimple element of $\bH$. Then $\Cent_{ \bH}(s)$ is reductive but may not be connected. The component group $A_\bH(s):=\Cent_{ \bH}(s)/\Cent_{ \bH}^\circ(s)$ is abelian and can be seen as a subgroup of $Z(\bH_0)$ via the homomorphism $$\omega_s\colon \Cent_{ \bH}(s)\to \Z(\bH_0) \text{  defined by  }g\mapsto [\dot{g},\dot{s}]\text{ for }\dot{g}\in\pi\inv(g),\ \dot{s}\in\pi\inv(s)$$ (see \cite[8.2]{Cedric}). 
	
	In type $\mathsf{D}$, we denote by $$\uE =\spann<F_p,  \gamma>$$  the subgroup of $E$ generated by $F_p$ and some $\gamma\in\Gamma$ of order 2 (essentially in the case of type $\mathsf{D}_4$ a choice is made and the graph automorphism of order 3 is left out). 
	Note that $\uE$ acts on both $\bH_0$ and $\bH$. We constantly consider the semidirect product $\bH\rtimes \uE=\bH\uE$. For $\sigma\in \HE\ \setminus\ \bH$ and $h\in\bH$ we occasionally write $\sigma(h)$ for $\sigma h\sigma\inv\in \bH$.
	We also define the submonoid $$\uE^+:=\{F_p^i\circ \gamma^j\mid i\geq 1,\ j\in\{0,1\}\}$$ whose elements can be considered as Frobenius endomorphisms $F\colon \bH\to\bH$ with fixed points forming the finite group $\bH^F:=\{h\in \bH\mid F(h)=h  \}$. For $x\in\bH^F$ ($F\in\uE^+$) we denote by $[x]_\HF$ its conjugacy class in $\HF$.
	
	In the following theorems, $\bH=\mathsf{D}_{l,\mathrm{ad}}(\FF)$ and $\bH_0=\mathsf{D}_{l,\mathrm{sc}}(\FF)$ are groups of rank $l\geq 4$ over an algebraically closed field $\FF$ of characteristic $p\neq  2$. The odd characteristic ensures that $\mathsf{D}_{l,\mathrm{ad}}(\FF)$ is actually the quotient $\bH_0/\Z(\bH_0)$ (see \cite[2.4.4]{DM21}).
	
	\begin{theo} \label{ThA} Let $F_0\in\uE^+$, $s\in \bH^{F_0}$ a semisimple element. Then there is a semidirect product decomposition  $$\Cent_\bH(s)=\Cent_\bH^\circ(s)\rtimes \wc A$$ with $F_0(\wc A)=\wc A$. 
		
		Moreover,
		if $\si '\in \uE$, $F:=F_0^k$ for some $k\geq 1$ and ${\si'}(s)\in [s]_\HF$, then there exists $\si \in \bH^F\si'\cap \Cent_{\bH\uE}(s)$ such that $\si(\wc A)=\wc A$ and $[{F_0,\sigma}]\in \wc A$. 
		
	\end{theo}
	
	The considerations to prove the following are more classical, using only the simply connected covering $\pi\colon\bH_0\to\bH$.
	
	\begin{theo} \label{ThB}
		Let $\calC$ be the $\bH$-conjugacy class of a semisimple element $t\in\bH$ such that $|A_\bH(t)|=4$. Let $F_0\in \uE^+$ and assume $F_0(\cC)=\gamma (\cC)=\cC$. Then $\pi(\bH_0^{F_0})\cap \calC\neq \emptyset$ and there exists $s\in \pi(\bH_0^{F_0})\cap \calC$ such that $\gamma(s)\in [s]_{\bH^{F_0}}$.
	\end{theo}

	Some more technical corollaries will be given in our last section.

	\section{Lifting of component groups and automorphisms}\label{ssec2C}
	
	We continue with the same notations about $\bH$, the adjoint group of type $\mathsf{D}_l$ ($l\geq 4$) over an algebraically closed field $\FF$ of odd characteristic, its maximal torus $\bT=\pi(\bT_0)$ and the abstract automorphism group $\uE$.

	\begin{prop}\label{propcheckW}
		There exists an $\uE$-stable subgroup $\wc W\leq \bH$ such that $\NNN_\bH(\bT)=\bT\rtimes \wc W$. 
\end{prop}
\begin{proof} The statement is probably known to experts familiar with adjoint groups (see \cite[Th. 1]{G17}). We recall an elementary construction of $\wc W$ and emphasize the stability property. For $d\geq 1$ let $J_d\in\GL_d(\FF)$ be the permutation matrix corresponding to the product of transpositions $(1,d)(2,d-1)\dots\in \Sym_d$. 
		
		Let $\bG:=\SO_{2l}(\FF)=\O_{2l}(\FF)\cap\SL_{2l}(\FF)$, where $\O_{2l}(\FF):=\{x\in \GL_{2l}(\FF)\mid {}^txJ_{2l}x=J_{2l}  \}$. Then $\bG$ is connected simple of type $\mathsf{D}_l$ (see for instance \cite[1.5.5]{DiMi}) with center $\{\Id_{2l}, -\Id_{2l}\}$ and $\bH =\bG/\{\Id_{2l}, -\Id_{2l}\}$. A maximal torus of $\bG$ is the group $\ov\bT$ of diagonal matrices whose diagonal is $(t_1,\dots, t_l,t_l^{-1},\dots , t_1^{-1})$ with $t_1,\dots, t_l\in\FF^\times$; a Borel subgroup $\ov\bB$ containing it and consisting of upper diagonal matrices is also described in \cite[1.5.5]{DiMi}. Numbering the basis in dimension $2l$ by $(1,2,\dots, l,-l,-l+1,\dots ,-2,-1)$ the normalizer of $\ov\bT $ in GL$_{2l}(\FF)$ is generated by $\ov\bT $ and the permutation matrices corresponding to $(l,-l)$ and the various involutions $(i,j)(-i,-j)$ ($1\leq i,j\leq l$). All those permutation matrices clearly belong to $\O_{2l}(\FF)$ and generate a subgroup $\ov V\leq \O_{2l}(\FF)$ with $\ov V\cong \ZZ/2\ZZ\wr \Sym_l$. The intersection $\ov W:=\ov V\cap\SL_{2l}(\FF)$ corresponds to elements of the wreath product whose number of non-trivial coordinates in the base group is even, thus forming a Coxeter group of type $\mathsf{D}_l$. 
		
		In $\bG=\SO_{2l}(\FF)$ we indeed get N$_\bG(\ov\bT)=\ov\bT \rtimes \ov W$. The elements of $\ov V$ are clearly fixed by $F_p$ which here is just the raising of matrix entries to the $p$-th power. On the other hand, conjugation by the permutation matrix associated to $(l,-l)$ induces the graph automorphism of $\bG$ of order 2 associated to $\ov\bT$ and $\ov\bB$, and swapping the fundamental reflections corresponding to $(l-1,l)(-l+1,-l)$ and $(l-1,-l)(-l+1,l)$ (see \cite[4.3.6]{DiMi}). This too preserves $\ov W$ since the permutation matrix associated to $(l,-l)$ belongs to $\ov V$. We get our Proposition by taking for $\wc W$ the image of $\ov W$ in $\bH =\bG/\{\Id_{2l}, -\Id_{2l}\}$.
\end{proof}

	\begin{rem} In type $\mathsf{B}_l$ the same statement is proved similarly in $\bH =\SO_{2l+1}(\FF)$ seen as $\O_{2l+1}(\FF)/\{\Id_{2l+1}, -\Id_{2l+1}\}$ (see again \cite{G17}). In type $\mathsf{A}_{l-1}$ with $\bH=\GL_{l}(\FF)/\{ \lambda\Id_{l}\mid\lambda\in\FF^\times \}$, the permutation matrices also provide a complement whose elements are fixed under $F_p$. Note that in that type the non-trivial element of $\Gamma$ sends our group of permutation matrices to a distinct $\NNN_{\bH }(\bT)$-conjugate in odd characteristic. Note also that permutation matrices have a somewhat cumbersome expression in terms of Chevalley generators. For instance the permutation matrix for the transposition $(1,2)$ has a class mod $\Z(\GL_l(\FF))$ corresponding to $n_{\al_1}(1)h_{\al_1}(-\omega)h_{\al_2}(-\omega^2)\dots h_{\al_{l-1}}(-\omega^{l-1})\in \SL_l(\FF)$ for some $\omega$ with $\omega^l=-1$ in the notations of \cite[1.12.1]{GLS3}, a similarly complex formula being necessary in type $\mathsf{D}$.
		
		In type $\mathsf{C}_l$, or exceptional types $\neq  \mathsf{G}_2$ it is known that $\bT$ has no complement in $\NNN_{\bH }(\bT)$ (see \cite[4.11]{AH17}). 
	\end{rem}

	We now consider the centralizer of some $x\in \bT$, where $\bT$ is as before a maximal torus of an adjoint group that is maximally split for any Frobenius endomorphism belonging to $\uE^+$. 
	
	Recall that for any $x\in \bT$ the connected centralizer $\Cent_{ \bH }^\circ(x)$ is a reductive group containing $\bT$ as maximal torus (see for instance \cite[3.5.4]{Carter2}).

	\begin{prop}\label{prop_wcB}
		Let $\bT$ and $\wc W$ be as in \Cref{propcheckW}, noting that $\wc W$ acts on $\bT$ hence naturally on the set $\Phi$ of $\bT$-roots of $\bH$. Let $x\in \bT$,
		$\Phi_x\subseteq X(\bT)$ the  root system  of $\Cent_\bH^\circ (x)$ \wrt $\bT$
		and choose $\Delta_x$ a basis of $\Phi_x$.
		We set $\wc B:=\Cent_{\bH}(x)\cap \wc W_{\Delta_x}=\{w\in \wc W\mid {}^wx=x,\ w(\Delta_x)=\Delta_x  \}$ and similarly $\wc B_\uE:=\Cent_{\HE}(x)\cap{(\wc W\rtimes \uE )}_{\Delta_x}$. 
		Then 
		\begin{thmlist}
			\item  
			$\Cent_{\bH }(x)= \Cent_\bH^\circ (x) \rtimes \wc B$, $\Cent_{\bH \uE }(x)= \Cent_\bH^\circ (x) \rtimes  \wc B_\uE $ ,
			\item $\wc B\lhd \wc B_{\uE }$ with abelian quotient. 
		\end{thmlist}
	\end{prop}
	\begin{proof} Note first that $\uE$ stabilizes $\bT$ and sends maximal tori of $\bH$ to maximal tori since $\si(^g\bT)={}^{\si(g)}\bT$ for any $g\in \bH$ and $\si\in\uE$. So $\bH\uE$ permutes the maximal tori of $\bH$. The connected centralizer $\bC:=\Cent_\bH^\circ (x)$ is the reductive group containing $\bT$ with root system $\Phi_x$, so 
		\begin{eqnarray*}
			\NNN_{\bH \uE }(\bC)&=&\bC.\NNN_{\bH \uE }(\bT,\bC)\ \ \text{ by conjugacy of maximal tori in $\bC$} \\
			&=&\bC.\NNN_{\wc W\uE}(\bC) \ \  \  \ \ \text{ by Proposition~\ref{propcheckW}} \\
			&=&\bC\rtimes (\wc W\uE)_{\Delta_x} \ \ \text{ by the regular action of $\NNN_{\bC }(\bT)/\bT$ on the bases of $\Phi_x$.} \end{eqnarray*}{}
		
		On the other hand $\bC$ is obviously normal in $\Cent_{ \bH \uE }(x)$, so $\Cent_{ \bH \uE }(x)$ is the centralizer of $x$ in the above $\bC\rtimes (\wc W\uE)_{\Delta_x}$. This gives the second claim of (a) with $\wc B_\uE:=\Cent_{\wc W  \uE }(x)_{\Delta_x}$. The first is clear by taking the intersection with $\bH$, which also gives the first part of (b). The quotient $\wc B_\uE/\wc B$ is abelian since it injects into $\uE$.
	\end{proof}

	\begin{rem} It could be interesting to classify those groups $\Cent_{ \bH \uE }(x)/\Cent_{ \bH }^\circ(x)$ or equivalently $\wc B_\uE$. 
		
		Using the simply connected covering $\bH_0$ we can embed $\Cent_{ \bH \uE }(x)/\Cent_{ \bH }^\circ(x)$ into $\ZZ\times V\times \ZZ/2\ZZ$ where $V$ is the dihedral group of order 8. Let us explain briefly how this can be seen.

		The simply connected covering $\bH_0\xrightarrow{\pi} \bH$ has kernel $\ZZZ(\bH_0)$ of order 4 while the covering $\bH_0\to \SO_{2l}(\FF)$ has a kernel $\{1,h_0\}$ of order 2 with $F_p(h_0)=\gamma(h_0)=h_0$ (see \cite[p. 70]{GLS3} where $h_0$ is called $z_c$).  The group $\bH $ acts on $\bH_0$ by conjugacy, with $\uE$ also acting in a compatible way, so we get an action of $\bH\uE$ on $\bH_0$ and therefore an action of $\Cent_{ \bH \uE }(s)$ on $\pi\inv (s)$. Denote $\pr_1\colon \Cent_{ \bH \uE }(s)\to \Sym_{\pi\inv (s)}$ the induced group morphism with values in a permutation group on four elements. Note that the kernel of the action of $\Cent_{ \bH  }(s)$ on $\pi\inv (s)$ is $\Cent_{ \bH  }^\circ(s)$ since this is $\pi \big(\Cent_{ \bH_0 }(\pi\inv(s))\big)$ by connectedness of centralizers of semisimple elements in $\bH_0$ (see \cite[3.5.6]{Carter2}). Moreover since the action of $\Cent_{ \bH \uE }(s)$ on $\pi\inv(s)$ commutes with translation by $h_0$, the image of $\pr_1$ is included in the centralizer of a product of two disjoint transpositions in the symmetric group on four elements. The latter is a dihedral group $V$ of order 8. 
		
		 Letting $\pr_2\colon \bH\rtimes \uE\to\uE$ be the projection on the second term of the semidirect product, we get that the kernel of the map $\Cent_{ \bH \uE }(s)\to V \times \uE$ defined by $c\mapsto (\pr_1(c),\pr_2(c))$ is $\Cent^\circ_{ \bH  }(s)$. Restricting this morphism to  $\wc B_{\uE}$ then makes it injective thanks to \Cref{prop_wcB}(a).
		\end{rem}
		 

Our applications to centralizers of semisimple elements and semisimple conjugacy classes will stem from the omnibus lemma below. 

\begin{lem}\label{omni} Let $F_0\in\uE^+$. Let $s\in\bH^{F_0}$ be a semisimple element and let $x\in \bT$ be an $\bH$-conjugate of $s$. Let $\Cent_{ \bH \uE }(x)=\Cent_{\bH }^\circ(x)\rtimes \wc B_\uE$ the decomposition of Proposition~\ref{prop_wcB}. Then there is an inner automorphism $\iota\colon\bH\uE\to \bH\uE$ induced by an element of $\bH$ such that   \begin{thmlist}
		\item $\iota(s)=x$,\item $\iota(F_0) \in \wc B_\uE$ and \item if $F=F_0^i$ for some $i\geq 1$, $\tau\in \uE$ and $h\in\bH^F$ are such that $\tau(s)=s^h$, then $$\iota(h\tau) \in  \Cent_{\bH }^\circ(x)\Cent_{\wc B_\uE}(\iota(F)).$$
		\end{thmlist}
\end{lem}

	\begin{proof} We denote $x:=s^g$ for some $g\in \bH$ and $\iota':\bH  \uE  \lra \bH  \uE $ given by $y\mapsto y^g$. Note that $\iota '$ preserves cosets $\bH e$ ($e\in \HE$). We have 
		 $\iota'({F_0})\in \bH {F_0}\cap \Cent_{\bH \uE}(x)$ since $s$ is $F_0$-invariant. 
		 Recalling $\wc B_\uE $ the group associated to $x$ from \Cref{prop_wcB} and the decomposition $ \Cent_{\bH \uE}(x)= \Cent^\circ _\bH(x)\wc B_\uE$ we can write $\iota'({F_0})$ as $cb$ with $c\in \Cent^\circ_\bH(x)$ and $b \in \wc B_\uE $. Lang's theorem (see for instance \cite[4.1.2]{Ge}) applied to $\Cent^\circ_\bH(x)$ and $\iota '({F_0})$ allows us to write $c\inv =c'{}\inv \iota '({F_0}){c'}\iota '({F_0})\inv$ in $\bH\uE$ for some $c'\in \Cent^\circ_\bH(x)$. Calling now $\iota: \bH  \uE \lra \bH  \uE$ the isomorphism given by  $y\mapsto y^{gc'}$, we get $$b=c\inv \iota '({F_0})=c'{}\inv \iota '({F_0}){c'}=\iota '({F_0})^{c'}=\iota({F_0})$$ and $\iota(s)=s^{gc'}=x^{c'}=x$.
		 
		 It remains to check (c).
		We have $\tau\in\uE$ and $h\in \bH^{F }$ such that $h\tau \in \Cent_{\bH\uE}(s)$. Then $\iota(h\tau)\in \Cent_{\bH\uE}(\iota(s)) = \Cent_{\bH\uE}(x)= \Cent^\circ_{\bH }(x)\rtimes \wc B_\uE$ by \Cref{prop_wcB}(a). So $$\iota(h\tau)\in \Cent^\circ_{\bH }(x)b'$$ for some $b'\in \wc B_\uE$.  
		 Note that $h \tau $ commutes with $F $ and accordingly $\iota(h \tau )$ commutes with $\iota(F )$. Because of the semidirect product structure $\Cent_{\bH\uE}(x)=\Cent_\bH^\circ(x) \rtimes \wc B_\uE$ and since $\iota(F )=\iota(F_0)^i\in \wc B_\uE$ thanks to (b), we see that $b'$ commutes with $\iota(F )$ as well. This gives our point (c).
	\end{proof}

	\begin{proof}[Proof of \Cref{ThA}]
	By Lemma~\ref{omni}, we have an inner automorphism $\iota$ of $\bH\uE$ such that $x:=\iota (s)\in \bT$, and $\Cent_{ \bH \uE }(x)=\Cent_{ \bH  }^\circ(x)\rtimes \wc B_\uE$ with $\iota(F_0)\in\wc B_\uE$. We define $\wc A:=\iota\inv(\wc B )$, recalling $\wc B =\wc B_\uE\cap \bH$ from \Cref{prop_wcB}. It is $F_0$-stable because $\wc B_\uE\cap \bH$ is stable under $\iota(F_0)$ since the latter belongs to $\wc B_\uE$. 
	
	We now consider some $\sigma '\in\uE$ such that $\sigma '(s)\in [s]_{\bH^{F}}$ for some $F=F_0^k$, $k\geq 1$. Let $h\in\HF$ such that $\sigma '(s)=s^h$. By \Cref{omni}(c) applied to $\tau:=\sigma '$, we get some $b\in\Cent_{\wc B_\uE}(\iota(F))$ such that $\iota(h\tau) \in \Cent_{\bH }^\circ(x)b$. We show that $\si:=\iota^{-1}(b)$ satisfies our claims. First
$\si$ centralizes $s$ since $b$ centralizes $x$. Also $\si$ belongs to $\bH\tau$ since 
$\iota$ and its inverse, being conjugations, are the identity on $\bH\uE/\bH$ and therefore $\bH \si =\bH b=\bH\tau $. Moreover $\si$ commutes with $F$ since $b$ commutes with $\iota(F)$, so $\si\in (\bH\tau)^F=\HF\tau$. The group $\wc A$ is $\si$-stable by the same argument used for $F_0$, namely $b\in\wc B_\uE$ and therefore $\iota\inv(b)$ normalizes $\wc A=\iota\inv(\wc B)$. That $b\in\wc B_\uE$ also implies that $[\iota(F_0),b]\in \wc B$ since $\wc B_\uE/\wc B$ is abelian by Proposition~\ref{prop_wcB}(b). This in turn tells us that $[F_0,\si]\in\wc A$, thus completing the proof of Theorem A. 
\end{proof}

In the next section, Theorem A will be applied mostly with $\sigma '$ being a graph automorphism. But here is a case reminiscent of
Shintani descent with two commuting Frobenius endomorphisms.

\begin{cor}\label{cor2.6}
Let $F,F_0\in \uE^+$, let $\cC$ be a semisimple conjugacy class of $\bH$ and assume we have $s_0\in \cC\cap\bH^{F_0}$  such that $F(s_0)\in [s_0]_\HFnull$. Applying \Cref{ThA} with $(F_0,s,\si ',k)$ being here $(F_0,s_0,F,1)$, we get a group $\wc A$ and some $\si\in\bH\uE$ that we denote as $\wc A(s_0)$ and $F'$.
 Then there exist $s\in \cC^F$, $F_0'\in \HF F_0$ centralizing $s$, and a group $\wc A(s)$ such that 
\begin{asslist}
\item $\Cent_\bH(s)=\Cent_\bH^\circ (s) \rtimes \wc A(s)$;
\item $\wc A(s)$ is $\spa{F,F_0'}$-stable; and there exists an inner isomorphism  $\iota_{s_0,s}: \bH \underline E \ra  \bH \underline E
$ such that 
\begin{itemize}
	\item $\iota_{s_0,s}(s_0)=s$\ \ and\ \ $\iota_{s_0,s}(\wc A(s_0))=\wc A(s)$;
	\item  $\iota_{s_0,s}(F_0)=F_0'$ \ and \ $\iota_{s_0,s}(F')=F$.
\end{itemize}
\end{asslist}
\end{cor}
\begin{proof} 
Let us recall how Theorem A is proved by application of \Cref{omni}. Choosing $x\in \bT\cap \cC$, we have $\Cent_{ \bH  }(x)=\Cent_{ \bH }^\circ (x)\rtimes \wc B$ from  \Cref{prop_wcB}. \Cref{omni} gives us an inner automorphism $\iota\colon \bH\uE\to\bH\uE$ with $\iota(s_0)=x$, $\iota(F_0):=f_0\in\wc B$, and we define $\wc A(s_0):=\iota\inv(\wc B)$. 

Having $F(s_0)\in[s_0]_\HFnull$ implies that ${[hF,s_0]}=1$ for some $h\in\HFnull$, so that $\iota(hF)\in\Cent_{ \bH  }^\circ(x)f$ with $f\in \Cent_{\wc B_\uE}(f_0)$ thanks to \Cref{omni}(c) for $i=1$ and $\tau=F$. Following the proof of Theorem A, we define $F':=\iota\inv(f)\in\Cent_{ \bH^{F_0}F  }(s_0) $ and get $F'(\wc A(s_0))=\wc A(s_0)$.

Recalling that $\iota$ is an inner automorphism and $\bH\uE/\bH$ is abelian, we have $\bH f=\bH F'=\bH F$ and Lang's theorem ensures the existence of some inner automorphism $\iota '$ of $\bH\uE$ induced by an element of $\bH$ such that $\iota'(f)=F$. We define $\iota_{s_0,s}:=\iota'\iota$, $F'_0:=\iota_{s_0,s}(F_0)=\iota'(f_0)$, $s:=\iota_{s_0,s}(s_0)=\iota'(x)$. The latter belongs to $\cC^F$ since $\iota$ and $\iota'$ are induced by elements of $\bH$ and $ [F,s]=\iota_{s_0,s}([\iota\inv(f),s_0])=1$. We have $\iota_{s_0,s}(F')=\iota'(f)=F$. We get $\Cent_{ \bH  }(s)=\Cent_{ \bH  }^\circ(s)\rtimes\wc A(s)$ for $\wc A(s):=\iota_{s_0,s}(\wc A(s_0))=\iota'(\wc B)$ as a consequence of the same property of $\wc A(s_0)$ for $s_0$. Moreover $\wc A(s)$ is $\spa{F,F'_0}$-stable since $\wc A(s_0)$ has been seen to be $F_0$-stable and $\wc B$ is $f$-stable since $f\in\wc B_\uE$. We also have $F'_0\in\bH F_0$ while $[F'_0,F]=\iota'([f_0, f])=1   $ as seen above. So $F'_0\in\bH^F F_0$ as claimed. We now have all our claims.
\end{proof}
	\section{Proof of Theorem B and some corollaries}

	We prove Theorem B but also some more technical statements, that can be seen as strengthenings of Theorem A, about semisimple classes of $\bH$ and the action of elements of $\uE^+$.
	 Recall $\gamma\in\uE$ (see \S 1) the graph automorphism stabilizing our choice of a maximal torus and Borel subgroup.

	 In the following $F_0\in\uE^+$ and $\cC$ is the $\bH$-conjugacy class of a semisimple element.
	
	\begin{cor}\label{1stCase}
		Assume $F_0(\cC)=\gamma (\cC)=\cC$ and there is some $s\in\cC^{F_0}$ such that $\gamma(s) \in[s]_{\bH^{F_0}}$. Then there  is $\gamma' \in \bH^{F_0}\gamma\cap \Cent_{\bH\uE}(s)$ and a $\spa{F_0, \gamma '}$-invariant subgroup $\wc A$ with $\Cent_\bH(s)=\Cent_\bH^\circ(s)\rtimes \wc A$.
	\end{cor}

	Recall that $\pi\colon \bH_0\to \bH$ is a composition $\bH_0\to \SO_{2l}(\FF)\to\bH$. The kernel of the second map is $\{\pm\Id_{2l}  \}$ while the kernel of the first is $\{1,h_0\}$ where $h_0\in\Z(\bH_0)$ is of order 2. 
	
	We keep $\cC$ a semisimple conjugacy class of $\bH$ and $F_0\in \uE^+$. We describe the situation complementary to the one of Corollary~\ref{1stCase}, but while the latter is a direct consequence of Theorem A, the following will require a bit more work.
	
	\begin{cor} \label{2ndCase} Assume that $F_0(\cC)=\gamma (\cC)=\cC$ but $\gamma(x) \notin[x]_{\bH^{F_0}}$ for every $x\in\cC^{F_0}$. 
		
	\noindent	 Let $F:=F_0^{2k}$ for some $k\geq 1$ and let $s\in \cC^F$.  
		 
	\noindent	Then there exist $F_0'\in\Cent_{\HF F_0}(s)$, $\gamma'\in\Cent_{\HF \gamma}(s)$ and an $\spa{F,F_0',\gamma'}$-stable group $\wc A$ such that 
		\begin{asslist}
			\item $\Cent_\bH(s)= \Cent^\circ_\bH(s)\rtimes \wc A $;
			\item $\omega_s(\wc A)=\spannh$ and $|\AHs|=|\wc A|
			=2$;
			\item $[s_0,F'_0]\in \Z(\bH_0)\setminus \spannh$  for every $s_0\in\pi\inv(s)$ and
			\item $[F_0', \gamma']=\wc a$, where $\wc a$ is the generator of $\wc A$.
		\end{asslist}
	\end{cor}

	\medskip
	
	Let $\cC$ be a semisimple conjugacy class of $\bH$ and $F'\in \bH\uE^+$ (seen as an endomorphism of $\bH$) such that $F'(\cC)=\cC$. Thanks to Lang's Theorem we have $\cC^{F'}\neq\emptyset$.
	
Taking $s\in\cC^{F'}$ and recalling that $A_\bH(s)$ is abelian, we can combine the parametrization of $\bH^{F'}$-conjugacy classes in $\cC^{F'}$ by $A_\bH(s)_{F'}:=A_\bH(s)/[A_\bH(s),F']$ (see \cite[4.3.6]{Ge}) and the injection $A_\bH(s)\to\Z(\bH_0)$ induced by the map $\omega_s$ from our introduction.
	\begin{eqnarray*}
		\cC^{F'}/\bH^{F'}\text{-conj}\ \ \stackrel{\sim}{\longleftrightarrow} & A_\bH(s)_{F'}&  \\
		& A_\bH(s)&\hookrightarrow\ \  \Z(\bH_0)
	\end{eqnarray*}
	An easy calculation then yields that the resulting map  
	\begin{eqnarray*}
		\Theta_{F'}\ \colon\	\cC^{F'}/\bH^{F'}\text{-conj}& \hookrightarrow & \Z(\bH_0)_{F'} \ \ \text{ is injective and defined by } \\
		\ 	[x]_{\bH^{F'}} & \mapsto & x_0\inv F'(x_0) [\Z(\bH_0),F']\ \ \text{ for any } x_0\in\pi\inv(x).
	\end{eqnarray*}

 Note that, in contrast to the other two maps, $\Theta_{F'}$ is independent of the choice of $s$ in $\cC^{F'}$. 
Note also that if moreover $F'\in \uE^+$ and denoting by $\uE_\cC$ the stabilizer of $\cC$ in $\uE$, then $\Theta_{F'}$ is $\uE_\cC$-equivariant by its definition (since $\uE$ commutes with $F'$). 	
	
\medskip

Let us show Theorem B (see also the third paragraph of the proof of \cite[6.17]{T14} for related considerations). We consider the above with $F'=F_0\in \uE$ and $s$ being denoted by $t$. We are assuming that
	$|A_\bH(t)|=4=|\Z(\bH_0)|$. The injection $A_\bH(t)\hookrightarrow\Z(\bH_0)$ is then an $\spa{F_0}$-equivariant isomorphism and therefore $\Theta_{F_0}$ is onto. So $1=x_0\inv F_0(x_0) [\Z(\bH_0),F_0]$ for some $x_0\in\bH_0$ with $\pi(x_0)\in\cC^{F_0}$. Rewriting this as $x_0\inv F_0(x_0)z\inv F_0(z)=1$ for some $z\in\Z(\bH_0)$, we get $x'_0:=x_0z\in \bH_0^{F_0}$ and
	$\pi(x_0')=\pi(x_0)\in\cC$. Note also that the equivariance and injectivity of $\Theta_{F_0}$ imply that $[\pi(x_0)]_{\bH^{F_0}}$ is $\gamma$-stable. This shows all claims of Theorem B.
	\medskip
	
	We now prove the corollaries. If there is some $s\in \cC^{F_0}$ such that ${[s]_{\bH^{F_0}}}$ is $\gamma$-stable then Theorem A with $k=1$ and $\sigma '$ being here $\gamma$ clearly gives the claims of Corollary~\ref{1stCase}.
	
\medskip
\noindent{\it Proof of Corollary~\ref{2ndCase}.}{}
The action of $\uE$ on $\Z(\bH_0)$ is easy to deduce from the description of $\Z(\bH_0)$ (see  \cite[Table 1.12.6]{GLS3}, \cite[2.9]{TypeD1}). In particular $\Cent_{\Z(\bHnull)}(\gamma)=\spa{h_0}$, a subgroup of ${\Z(\bHnull)}$ of order 2.

	By assumption, no $\HFnull$-class contained in $\cC^{F_0}$ is $\gamma$-stable, so the $\HFnull$-conjugacy classes contained in $\cC$ are mapped via $\Theta_{F_0}$ to elements in $\Z(\bH_0)_{F_0}$ not fixed by $\gamma$. Then $\Z(\bH_0)_{F_0}$ can't have order 1 or 2, so it implies the following\begin{lem}\label{3.3} Keeping the assumptions of \Cref{2ndCase},
		$F_0$, and therefore also $F$, act trivially on $\Z(\bH_0)$.
	\end{lem}  We also get  that the image of $\Theta_{F_0}$, being $\gamma$-stable, is indeed equal to  the whole $\Z(\bH_0)\setminus\spa{h_0}$. This is of order 2, so $|A_\bH(s)_{F_0}|=|A_\bH(s)_{}|=2$ for any $s\in\cC^{F_0}$ by the diagram defining  $\Theta_{F_0}$ from $\omega_s$. The equality $|A_\bH(s)_{}|=2$ of course holds for any $s\in\cC$. We have $\gamma\omega_s\gamma =\omega_{\gamma(s)}$ and the image of $\omega_s$ depends only on the conjugacy class of $s$ (by the same formula for inner automorphisms), so $\gamma(\cC)=\cC$ implies that $\omega_s(A_\bH(s))$ is $\gamma$-stable of order 2 for any $s\in\cC$. Then indeed $\omega_s(A_\bH(s))=\spannh$, which amounts to part (ii) once the other parts are checked. 
	
We prove the other statements (i), (iii) and (iv) in several steps, first for some $s\in \cC^{F_0}$ and then for $t\in \cC^F$ via an application of Theorem A and Lemma~\ref{omni}. 
	
For the following we fix some $x\in \cC\cap \bT$ and the groups $\wc  B_\uE\leq \bH \uE$ and $\wc B:=\wc B_\uE \cap \bH$ with $\Cent_{ \bH \uE }(x)=\Cent_\bH^\circ (x)\rtimes \wc B_\uE$ from \Cref{prop_wcB}.

	Let us take $s\in \cC^{F_0}$. Let $s_0\in\pi\inv (s)$ and $h':=s_0\inv F_0(s_0)$. Note that $\Theta_{F_0}([s]_{\bH^{F_0}})=h'\in \Z(\bH_0)\setminus\spa{h_0}$ by the above. 
	This ensures $[s_0,F_0]\in \Z(\bH_0)\setminus \spannh$  for every $s_0\in\pi\inv(s)$. Defining $F'_0:=F_0$ in that case, we get part (iii) for that $s$.
	
	We turn to the question of actually finding $\gamma '$ in $\Cent_{ \bH \uE }(s)^F$ and proving (i) and (iv). In order to apply Theorem A with $\si '=\gamma$, we need to check that $[s]_\HF$ is $\gamma$-stable.
	
 Since $F=F_0^{2k}$ and $F_0(s_0)=s_0h'$ with $h'\in\Z(\bH_0)$, we have $$\Theta_F([s]_\HF)=s_0\inv F_0^{2k}(s_0)=s_0\inv F_0(s_0) F_0(s_0\inv F_0(s_0))\cdots F_0^{2k-1}(s_0\inv F_0(s_0))=(h')^{2k}$$ thanks to \Cref{3.3}. 
 Squares in $\Z(\bH_0)$ are $\uE$-fixed since they form a stable subgroup of order 1 or 2. So $\Theta_F([s]_\HF)$ is $\gamma$-fixed and by equivariance of $\Theta_F$, this implies that $[s]_\HF$ is $\gamma$-stable. Note that it is also $F_0$-stable since $F_0(s)=s$. On the other hand $\cC^F/\HF\text{-conj}$ is of cardinality $|A(s)_F|=|A(s)|=2$ as seen before and using 
 \Cref{3.3} again. Since $\cC^F$ is $\spa{F_0,\gamma}$-stable and $[s]_\HF$ is $\spa{F_0,\gamma}$-stable it implies the following \begin{lem}\label{3.4}
 Under the assumptions of \Cref{2ndCase}	$\spa{F_0,\gamma}${ acts trivially on   }$\cC^F/\HF\rm{-conj}$.
 \end{lem}

	 We can now apply Theorem A to $s$ with $i=2k$, $\sigma '=\gamma$. Taking also $F'_0=F_0$, this gives us a $\sigma$ that we call $\gamma '$ and a decomposition $\Cent_\bH(s)=\Cent_\bH^\circ(s)\rtimes \wc A(s)$ satisfying our part (i) and $[F_0,\gamma ']\in\wc A(s)$. 
	 
	 Let us check that $[F_0,\gamma ']\neq 1$. Remembering that $\bH\uE$ acts on $\bH_0$, it suffices to show that ${[\gamma ',F_0]}(s_0)\neq s_0$. Recall that $F_0(s_0)=s_0h'$ for some $h'\in \Z(\bH_0)\setminus\spa{h_0}$ so that $\gamma '(h')=\gamma '{}\inv (h')=\gamma(h')=h_0h'$ since $\Z(\bH_0)$ is a group of order $4$ in which the centralizer of $\gamma$ is the subgroup of order $2$ generated by $h_0$. 
	 Let $z\in \Z(\bH_0)$ such that  $\gamma '(s_0)=s_0z$. Note that $z$ exists since $\gamma '(s)=s$.
	 Using also \Cref{3.3} we get 
	 \begin{align*}
	 {[\gamma ',F_0]}(s_0)&=\gamma 'F_0\gamma '{}\inv(s_0h'{}\inv)=\gamma 'F_0(s_0\gamma '{}\inv(z\inv)h_0h'{}\inv)=\\
	 &=\gamma '(s_0\gamma '{}\inv(z\inv)h_0)=s_0h_0\neq s_0.
	 \end{align*}
	 
	 This implies our claim that $[F_0,\gamma ']\neq 1$. (Another proof is possible with the three subgroups Lemma.) 
	 Since $\wc A(s)$ is a group of order 2 and $[F_0,\gamma ']\in \wc A(s)$, $[F_0,\gamma ']$ generates $\wc A(s)$. This gives (iv) in that case.
	 
	 Let us recall however how Theorem A is deduced from Lemma~\ref{omni} in our case (with $\si '$ being $\gamma$). We first get an inner automorphism of $\bH\uE$ induced by an element of $\bH$, namely $\iota_s\colon \bH\uE\to\bH\uE$ , such that $\iota_s(s)=x$, $f_s:=\iota_s(F_0)\in\wc B_\uE$ and then $\wc A(s):=\iota_s\inv(\wc B)$  satisfies $\Cent_{ \bH  }(s)=\Cent_\bH^\circ (s)\rtimes \wc A(s)$. Finally we take some $g_s\in \wc B_\uE\cap \bH^F\gamma$ and define $\gamma ':= \iota_s\inv(g_s)$.

	  Before turning to the general case of some $t\in\cC^F$, let us apply now $\iota_s$ to $F_0$ and $\gamma'$. From what has been said about $[F_0,\gamma ']$ we see that $[f_s,g_s]=\wc b$,  the generator of $\wc B$. Since $|\wc B|=|\AHs|=2$ and hence $\wc B\leq \Z(\wc B_\uE)$ this implies $$[f,g]=\wc b\text{ for every }f\in  \wc B_\uE\cap \bH F_0 \ (=\wc B f_s)\text{ and }g\in \wc B_\uE\cap \bH \gamma\ (=\wc B g_s). \eqno(3.5)$$ This is a general property since such an $s$ always exists.
	
	Let us now look at the general case where $t\in\cC^F\setminus \cC^{F_0}$. Using again Lemma~\ref{omni} but this time with $i=1$ (so that the $F_0$ of that lemma is our $F$), we get a conjugation $\iota_t\colon \bH\uE\to \bH\uE$ such that $\iota_t(t)=x$ and some other properties.
	
	We define $\wc A(t):=\iota_t\inv(\wc B)$ which then satisfies $\Cent_{ \bH  }(t)=\Cent_\bH^\circ (t)\rtimes \wc A(t)$. 
	Let $f:=f_s=\iota_s(F_0)\in \wc B_\uE\cap \bH F_0$ (see above), and take any $g\in  \wc B_\uE\cap \bH \gamma$. Let  $F_{0}':=\iota_t\inv(f)$ and $\gamma_{t}':=\iota_t\inv(g)$. By this construction $\wc A(t)$ is $\spa{F_{0}', \gamma_{t}'}$-stable and $F_{0}'(t)=\gamma_{t}'(t)=t$. 
	
	In the next step we show that $F$ centralizes in $\bH\uE$ both $\gamma'_t$ and $F'_0$. In applying \Cref{omni}, we can take $\tau\in\{F_0,\gamma\}$ and $h\in \bH^F$ such that $\tau(s)=s^h$ thanks to \Cref{3.4}. \Cref{omni}(c) tells us that $\iota_t(h \tau ) =c f'$ with  $c \in { \bH  } $, $f' \in {\wc B_\uE}$, and $[f',\iota_t(F)]=1$. Since $\iota_t(h\tau)\in \bH\tau$, this shows that $\iota_t(F)$ fixes an element of the set with two elements $\wc B_\uE\cap \bH\tau$. On the other hand $\iota_t(F)\in\wc B_\uE$ by \Cref{omni}(b), so it stabilizes $\wc B_\uE\cap \bH\tau$ (remember again
	that cosets mod $\bH$ are stabilized by inner automorphisms of $\bH\uE$) hence fixes any of its two elements. Taking the image by $\iota_t\inv$, we get that $F$ centralizes both $\iota_t\inv (\wc B_\uE\cap \bH F_0)$ and $\iota_t\inv (\wc B_\uE\cap \bH \gamma)$, hence both $F'_0$ and $\gamma '_t$. So we get that $F'_0\in \Cent_{\bH^F F_0}(s)$ and $\gamma '_t\in  \Cent_{\bH^F \gamma}(s)$. This ensures part (i) in that case.

Letting $t_0\in\pi\inv(t)$, $x_0\in\pi\inv(x)$, we have  $[t_0,F'_0]= \iota_t\inv( [x_0,f])= \iota_t\inv\iota_s( [s_0,F_0])=[s_0,F_0]$, since $\iota_s$ and $\iota_t$ are conjugations by elements of $\bH$ hence act trivially on $\Z(\bH_0)$. Then  $[t_0,F'_0]=  [s_0,F_0]\in \Z(\bH_0)\setminus \spannh$ as seen before. We then get (iii) for $t$.

	 Let now $\wc a$ be the generator of $\wc A(t)$ and hence $\wc a=\iota_t\inv (\wc b)$.  By (3.5), $[f,g]=\wc b$ and hence $[F_{0}', \gamma_{t}']=\iota_t\inv([f,g])= \iota_t\inv(\wc b)= \wc a$. This ensures part (iv) with $\gamma ':=\gamma'_t$.    	 
	 \qed
	 
\medskip\noindent{\bf Remark 3.6.}  
		Both cases covered by our Corollaries are non empty. It is clear for the first.
		
		 In order to check the relevance of the second corollary we use the notations of \cite[Not. 2.3]{TypeD1} to describe elements of the maximal torus $\bT_0$ of $\bH_0$.  Let $q$ be a power of $p$ with $q\equiv 1 \mod 4$. Let $F_q$ be the corresponding power of $F_p$. Let $\zeta, \varpi\in\FF^\times$ with $\zeta^{q^2-1}=\varpi^2=-1$. Recall that $\hh_\ul(\varpi):=\hh_{e_1}(\varpi)\dots \hh_{e_l}(\varpi)$ is an element of $\Z(\bH_0)\setminus\spa{h_0}$, not fixed by $\gamma$. Let $l=4$, $s_0=\hh_{e_2}(\varpi)\hh_{e_3}(\zeta)\hh_{e_4}(\zeta^q\varpi)\in \bT_0$. One has $\gamma(s_0)=s_0$. It is not difficult to see that with the equation $\zeta^{q^2}\varpi^q\varpi=\zeta$ one gets that $F_q(s_0)$ and $s_0\hh_{\underline 4}(\varpi)$ are $\NNN_{\bH _0}(\bT_0)$-conjugate, so we can write $F_q(s_0)=s_0^n\hh_{\underline 4}(\varpi)$ for some $n\in \NNN_{\bH _0}(\bT_0)$. Then $\Theta_{nF_q}([s]_{\bH^{nF_q}}) =\hh_{\underline 4}(\varpi)$ for $s:=\pi(s_0)$.
	
		By Lang's Theorem we can take $\iota\colon \bH\uE\to\bH\uE$ an inner automorphism induced by an element of $\bH$ such that $\iota(nF_q)=F_q$. Then $s':=\iota(s)\in \bH^{F_q}$ and $\Theta_{F_q}([s']_{\bH^{F_q}})=\hh_{\underline 4}(\varpi)$ which is not $\gamma$-stable. Then $[s']_{\bH^{F_q}}$ is not $\gamma$-stable while $s$ and $s'$ are in the same $\bH$-class.

	\printindex
\end{document}